\documentclass{article}

\usepackage{amsfonts} % Provides additional mathematical fonts
\usepackage{amsmath} % Enhances mathematical typesetting
\usepackage{amssymb} % Additional mathematical symbols
\usepackage{amsthm} % Defines theorem, lemma, corollary, and proof environments

\usepackage{enumerate} % Customizable enumeration environments
\usepackage{dutchcal} % Enables calligraphic math symbols

\usepackage{graphicx} % Required for inserting images
\usepackage{xcolor} % Allows text and background color customization
\usepackage{hyperref} % Enables hyperlinks in the document
\hypersetup{
    colorlinks=true,
    linkcolor=blue,
    filecolor=magenta,
    urlcolor=blue,
}
\usepackage{tikz}
\usetikzlibrary{matrix,arrows,positioning,decorations.markings,decorations.pathmorphing}

\usepackage{mathtools} % Improvements over amsmath, including \coloneqq
\usepackage[activate={true,nocompatibility},final,tracking=true,kerning=true,spacing=true,factor=1100,stretch=10,shrink=10]{microtype}
\microtypecontext{spacing=nonfrench}
\usepackage{authblk}  % Package for affiliations
\usepackage{orcidlink} % For ORCID icons 

\newtheorem{theorem}{Theorem}

\newtheorem{definition}[theorem]{Definition}

\newtheorem{lemma}[theorem]{Lemma}

\newtheorem{proposition}[theorem]{Proposition}
\newtheorem{remark}[theorem]{Remark}

\newcommand{\R}{\mathbb{R}}
\newcommand{\g}{\tilde{\mathfrak{g}}}
\newcommand{\kk}{\tilde{\mathfrak{k}}}

\newcommand{\A}{\mathcal{A}}
\newcommand{\vol}{\mathbf{v}}
\newcommand{\tp}{\tilde{p}}

\newcommand{\llangle}{\mathopen{\langle\!\langle}}
\newcommand{\rrangle}{\mathclose{\rangle\!\rangle}}

\title{Hamiltonian Reduction in Affine Principal Bundles}

\author[1]{M.\'A. Berbel \orcidlink{0000-0002-4269-9599}}  
\author[2]{M. Castrill\'on L\'opez \orcidlink{0000-0001-7673-9870}}

\affil[1]{Departamento de Matemática Aplicada, Universidad Pontificia Comillas,
Alberto Aguilera 25, 28015-Madrid, Spain\\
    \texttt{maberbel@comillas.edu}}
\affil[2]{Departamento de \'Algebra, Geometr\'\i a y Topolog\'\i a, Universidad Complutense de Madrid, Plaza de las Ciencias 3, 28040-Madrid, Spain\\
    \texttt{mcastri@mat.ucm.es}}
\date{}

\begin{document}

\maketitle

 \emph{Dedicated to professor M.C. Mu\~noz Lecanda, with admiration and gratitude}
 
\begin{abstract}
    This paper presents a Hamiltonian reduction procedure for field theories over affine principal bundles introducing a canonical identification to describe the reduced multisymplectic space without the introduction of a connection. The main goal is to provide a Hamiltonian analogue of the Lagrangian reduction theory developed in \cite{affinered}. The core of this work lies in the derivation of this canonical identification, the reduced Hamilton–Cartan equations, and a reduced covariant bracket that describes the dynamics. Finally, this theoretical framework is illustrated with a fundamental example: molecular strands.
    
\end{abstract}
\vspace{4mm}


\vspace{4mm}

\noindent \em Mathematics Subject Classiﬁcation 2020: \em 70S05, 58D19, 70G45, 70S10

\section{Introduction}

Reduction procedures in field theories with symmetries, either in the Lagrangian or in the Hamiltonian frameworks, provide a privileged tool to simplify problems and to understand some of their properties. With respect to the Hamiltonian side (on which we focus in this article), over the years, two primary approaches have emerged in the literature for tackling geometric reduction. The first approach involves the introduction of a (multi)momentum map (for example, see \cite{MultiGroupoids, RemarksMultiSymp, Madsen12, Madsen13}) with the aim to generalize the Marsden--Weinstein reduction Theorem within the multisymplectic formalism developed in \cite{multi_formalism, multisymplectic}. Another approach generalizes Poisson-Poincar\'e reduction by proposing a reduction scheme for a covariant bracket formulation.

The later is the case of \cite{someremarks}, where the covariant bracket formulation developed in \cite{PoissonForms, CovariantPoisson} is reduced for the particular case of a $K$-principal bundle $P\to M$. The key idea in \cite{someremarks} is to describe the reduction of the polysymplectic space $\Pi_P=TM\otimes V^*P\otimes \wedge ^n T^*M$ ($n=\mathrm{dim}M$) via the natural identification
\begin{equation} \label{eq: someremarks}
        \Pi_{P}/ K\cong T M \otimes \kk^{*} \otimes \bigwedge^{n} T^*M,
\end{equation}
where $\tilde{ \mathfrak{k}}$ is the adjoint bundle. 
Subsequently, these results were generalized in \cite{PoissonPoincare} to arbitrary bundles $P\to M$ in which a Lie group $K$ is acting vertically and properly. Then, $P\to P/K$ is a principal bundle and choosing a principal connection $\A$, the following identification provides a description of the reduced polsymplectic space,
\begin{equation} \label{eq: Poisson-Poincare}
        [\text{hor}^*_{\A}]\oplus [J]_G :\Pi_{P}/ K\cong_{\A}\left(T M \otimes \kk^{*} \otimes \bigwedge^{n} T^*M\right) \oplus \Pi_{P/K},
\end{equation}
where $\text{hor}^*_{\A}$ is the dual morphism to the horizontal lift provided by $\A$ and $J$ is the equivariant momentum map of the action on $T^*P$. However, identification \eqref{eq: Poisson-Poincare} is not canonical and depends essentially on the chosen auxiliary connection $\A$. This raises the question of whether there is a formulation of these reduction procedures that does not necessitate the introduction of external elements that may not have a clear physical interpretation.

Affine principle bundles are remarkably relevant examples in field theories constructed from a $K$-principal bundle $Q\to M$ and an associated vector bundle $E=(Q\times V)/K\to M$. These bundles can be used to describe charged molecular rods \cite{MolStrand} and other models using the $G$-strand building introduced in \cite{matrixgstrands,gstrands}, whenever $G$ is a semidirect product. %Algún otro ejemplo??
The group $K$ acts on the affine principal bundle $P=Q\times E$ and the Poisson-Poincar\'e reduction developed in \cite{PoissonPoincare} could be applied to reduce any $K$-invariant field theory on $P$. Yet, the relevance of these bundles (see also, \cite{AffineLieP}) justifies a specific study avoiding the introduction of any auxiliary connection in the process of reduction. This approach was explored in \cite{affinered}, where the Lagrangian reduction of field theories on affine principal bundles was studied without relying on a principal connection. The primary goal of the present paper is to develop a Hamiltonian reduction procedure in which the splitting \begin{equation} \label{eq: H1dentification}
        \Pi_{P}/ K\cong\left(T M \otimes \kk^{*} \otimes \bigwedge^{n} T^*M\right) \oplus \Pi_{E}
    \end{equation}
is canonical. The results obtained herein therefore comprise the Hamiltonian counterpart to the theory developed in \cite{affinered}.

The structure of this paper is as follows. Section \ref{sec: Lagrangian Affine} provides a brief introduction to Affine Principal bundles and the Lagrangian reduction of field theories defined over them. Section \ref{sec: multi} then reviews key aspects of Hamiltonian field theories, including multisymplectic and polysymplectic structures, as well as a covariant bracket formulation. In Section \ref{sec: covariant bracket}, we explore these concepts in the context of Affine principal bundles. The core of this work lies in Section \ref{sec: reduction}, where we establish the canonical identification \eqref{eq: H1dentification}, derive the reduced Hamilton–Cartan equation, and introduce a reduced covariant bracket that describes the dynamics. Finally, Section \ref{sec: example} demonstrates the practical relevance of our framework by applying it to a fundamental example: molecular strands.

\section{Lagrangian Reduction in Affine Principal Bundles}  \label{sec: Lagrangian Affine}

Let $\pi : Q\to M$ be a principal $K$-bundle, and consider a left linear representation of $K$ on a real vector space $V$. Let $E = (Q\times V)/K \to M$ be the associated vector bundle. We consider the affine group $G = K \ltimes V$ as the semidirect product with group operation
\[
(g, v) \cdot (g', v') = (gg', gv' + v), \quad \forall g, g' \in K, \forall v, v' \in V.
\]
The \em affine principal bundle defined by $Q\to M$ and the representation of $K$ on $V$ \em is the product bundle $P = Q \times_M E \to M$, which is a $G$-principal bundle, with right action
\[
R(g,v)(u_x, e_x) = (R_g u_x, e_x + [(u_x, v)]_K),
\]
for any $(u_x, e_x) \in P$, $(g, v) \in G$.

The group $K$ embeds as a closed subgroup of $G = K \ltimes V$ via the injection $g \in K \mapsto (g, 0)$. Similarly, the bundle $Q$ is a principal subbundle of $P = Q \times_M E$ via $u_x \in Q \mapsto (u_x, 0_x)$, where $0_x$ is the zero vector in the fiber $E_x$ of $E$.

Let $C(Q)$ be the connection bundle of the principal bundle $Q$. Then, the mapping

\begin{equation}\label{eq: Lidentification}
\begin{aligned}
    \Psi: \left(J^1(Q\times_M E)\right)/K &\to C(Q)\times_M J^1E, \\
    [j^1_x(s, e)]_K &\mapsto ([j^1_x s]_K, j^1_x e).
\end{aligned}
\end{equation}
is a fiber diffeomorphism. This identification is canonical, as it does not require a choice of a principal connection, and it plays a fundamental role in the formulation of Lagrangian field theories on affine principal bundles.

Let \( L\vol \) be a Lagrangian density over \( J^1(Q \times_M E) \) invariant under the action of the subgroup \( K \subset K \ltimes V \), that is, \(
R^{(1)}_g L = L \quad \text{for all } g \in K.
\)
Due to this invariance, the Lagrangian \( L \) descends to a reduced Lagrangian defined on the quotient space,
\[
l:\left(J^1(Q\times_M E)\right)/K = C(Q) \times_M J^1E \to \mathbb{R},
\]
where the identification \eqref{eq: Lidentification} is used. It is possible to formulate a new variational principle for \(l\) considering reduced sections and their corresponding reduced variations. Concretely, given a (local) section \( (s, e): U \to Q \times_M E \) over an open domain \( U \subset M \) and an infinitesimal variation \( (\delta s, \delta e) \in \Gamma(s^* V(Q) \oplus e^* V(E)) \) on it, the corresponding infinitesimal variation for the reduced section $\Psi\circ j^1(s,e)$ is
\[
\Psi_* \left[ j^1(\delta s, \delta e) \right]_K = (\delta \sigma, j^1 \delta e) \in \Gamma(\sigma^* V(C(Q)) \oplus (j^1 e)^* V(J^1 E)),
\]
where, as discussed in \cite[pp.~171--172]{EPFT2001},
\[
\delta \sigma = \nabla^\sigma \eta \in \Gamma(T^*M \otimes \mathrm{ad} Q) = \Gamma(\sigma^* V(C(Q))),
\]
and \( \eta \) is the unique section of the adjoint bundle \( \mathrm{ad} Q \) such that \( \delta s = \eta \circ \mathrm{Im}(s) \). This Lagrangian reduction on affine principal bundles is fully studied in \cite{affinered} and the core results are articulated in the following theorem:

\begin{theorem}[Lagrangian Reduction] \label{Lagrangian reduction}
Let \( L: J^1(Q \times_M E) \to \mathbb{R} \) be an \( K \)-invariant Lagrangian and let \( l: C(Q) \times_M J^1E \to \mathbb{R} \) be the reduced Lagrangian. For an open set \( U \subset M \), with \( \bar{U} \) compact, and a section \( (s, e): \bar{U} \to Q \times_M E \), define \( \sigma: \bar{U} \to C(Q) \) as in the identification \eqref{eq: Lidentification}. Then, the following are equivalent:

\begin{enumerate} [\em (i)\em]
    \item The variational principle
   \(
   \delta \int_U L(j^1 s, j^1 e) \vol = 0
   \)
   holds for vertical variations \( (\delta s, \delta e) \) along \( (s, e) \) with compact support.

   \item The local section \( (s, e) \) satisfies the Euler-Lagrange equations for \( L \).

   \item The variational principle
   \(
   \delta \int_U l(\sigma, j^1 e) \vol = 0
   \)
   holds using variations of the form:
   \[
   \delta \sigma = \nabla^\sigma \eta \in \Gamma(\sigma^* V(C(Q))), \quad \delta e \in \Gamma(e^* V(E)),
   \]
   where \( \eta \) is any section of the adjoint bundle \( \mathrm{ad} Q \) and \( \delta e \) any section of \( e^* V(E) \), both with compact support.

   \item The Lagrange-Poincar\'e equations hold:
\begin{equation} \label{eq: LP_equations}
    \mathcal{EL}(l(\sigma, \ )) = 0, \quad \mathrm{div}^\sigma \frac{\delta l(\  , j^1e)}{\delta \sigma} = 0,
\end{equation}
where \(\mathcal{EL}\) represents the Euler-Lagrange operator of the Lagrangian \( l(\sigma, \ ) : J^1E \to \mathbb{R} \), and 
\(\frac{\delta l( \ ,j^1e)}{\delta \sigma} \in \Gamma(\sigma^*V^*C(Q)) = \Gamma(TM \otimes (\mathrm{ad}Q)^*)\) 
is the vertical differential of \( l( \ , j^1e) : C(Q) \to \mathbb{R} \) restricted to vertical vectors along \( \sigma \). 
The operator \(\mathrm{div}^\sigma\) is defined by:
\begin{equation}
    \langle \mathrm{div}^\sigma D, \eta \rangle = -\langle D, \nabla^\sigma \eta \rangle + \mathrm{div} \langle D, \eta \rangle, 
    \quad \forall \eta \in \Gamma(\mathrm{ad}Q),
\end{equation}
where the bilinear products are the natural ones.
\end{enumerate}
\end{theorem}

The main aim of this paper is to obtain an analogous result to Theorem \ref{Lagrangian reduction} for Hamiltonian field theories over affine principal bundles.

\section{Multisymplectic and Polysymplectic bundles} \label{sec: multi}  

The \em dual jet bundle \em $J^1P^* $ of a fiber bundle $\pi_{M,P}:P\rightarrow M$ is a vector bundle over $P$ such that for any $y \in P_x,x\in M$, the fiber is $J^1_yP^*=\mathrm{Aff}(J^1_yP,\bigwedge^nT_x^*M)$, the set of affine maps from the jet bundle to the bundle of $n$-forms on $M$, where $n=\mathrm{dim} M$. Provided a fiber chart $(x^i,y^a)$ on $P$, we define fiber coordinates $(x^i, y^a, p^i_a, \tp)$ on $J^1P^*$ so that its elements have the expression
\begin{equation}
y_i^a\mapsto (\tp+p^i_ay^a_i)d^nx,
\end{equation}
where $d^nx=dx^1\wedge\cdots\wedge dx^n$. Observe that the use of affine morphisms implies that $\mathrm{rank}(J^1P^*)=\mathrm{rank}(J^1P)+1$. 

The multisymplectic space $J^1P^*$ is canonically isomorphic to the subbundle $Z$ of $\bigwedge^nT^*P$ of $n$-forms annihilated when contracted with two vertical vectors \cite{gotay}. This allows to define a \em canonical \em $n$-form $\Theta$ on $J^1P^*\cong Z$ defined as
\begin{equation}
\Theta(z)(u_1,\dots,u_n)=z\left(T\pi_{P,\bigwedge^nT_x^*P}u_1,\dots,T\pi_{P,\bigwedge^nT_x^*P}u_n\right), \quad z\in Z,
\end{equation}
where $u_i\in T_zZ$ for $i=1,\dots,n$, along with the \em canonical multisymplectic \em $(n+1)$-form
\begin{equation}
\Omega=-d\Theta.
\end{equation}
In local coordinates, these forms are expressed as follows:
\begin{equation}
\Theta=p_a^idy^a\wedge d^{n-1}x_i+\tp d^nx,
\end{equation}
\begin{equation}
\Omega=dy^a\wedge dp_a^i \wedge d^{n-1}x_i-d\tp\wedge d^nx.
\end{equation}

A more extensive analysis of the multisymplectic formalism is available in \cite{multi_formalism, multisymplectic, gotay}. Note that the multisymplectic bundle defined in this paper is a multisymplectic manifold. For an  analysis of how to study the dynamics of general multisymplectic manifolds, see \cite{RemarksMultiSymp}.

The \em polysymplectic bundle \em of a fiber bundle $\pi_{M,P}:P\rightarrow M$ is \begin{equation}
\Pi_P=TM\otimes V^*P\otimes\bigwedge^nT^*M.
\end{equation}
It is a vector bundle over $P$ with $\mathrm{rank}(\Pi_P)=\mathrm{rank}(J^1P)$ which is dual to the vector bundle $T^*M\otimes VP$ that models $J^1P$. In fact, the projection of the the affine map $\varphi\in J^1_yP^*=\mathrm{Aff}(J^1_yP,\bigwedge^nT_x^*M)$ into its associated linear map 
$\vec{\varphi}\in (\Pi_P)_y=(T_x^*M\otimes V_yP)^*\otimes\bigwedge^nT_x^*M$, is an affine bundle modeled on the rank-one vector bundle $\bigwedge^nT^*M\to M$ \cite{AffineBracket}. 

A \em Hamiltonian system \em is a pair $(\Pi_P,\delta)$ where $\delta$ is a section of $J^1P^*\to \Pi_P$. Given such section, the canonical form on $J^1P^*$ can be pulled back to $\Pi_P$ as $\Theta_{\delta}=\delta^*\Theta$ and $\Omega_\delta=-d\Theta_{\delta}$ to provide dynamics. Thus, a section $p$ of $\Pi_P\to M$ is a \em solution \em of the Hamiltonian system if for any vertical vector field $X$ on $\Pi_P$,
\begin{equation} \label{Hameq}
0=p^*i_X\Omega_\delta.
\end{equation}

For each $p\in \Pi_{P}$, there exist local coordinates $(x^i, y^a, p^i_a)$  in which $p$ can be expressed as
\begin{equation*}
p=p^i_a\frac{\partial}{\partial x^i}\otimes dy^a \otimes d^nx,
\end{equation*}
the section $\delta(x^i,y^a, p^i_a)=(x^i, y^a,p^i_a, H_{\delta}(x^i,y^a, p^i_a))$, and
\begin{equation}\label{pullback canform}
\Theta_{\delta}=p_a^idy^a\wedge d^{n-1}x_i+H_{\delta}(x^i,y^a, p^i_a)d^nx.
\end{equation}
In these coordinates, given $E:\Pi_P\to \R$, a $G$-invariant real function on $\Pi_P$,
\begin{equation}\label{eq: function inv}
\frac{\partial E}{\partial y^{\alpha}}=\frac{1}{2}p^i_{\gamma}c^{\gamma}_{\beta\alpha}\frac{\partial E}{\partial p^i_\beta}.
\end{equation}
Similarly, for any $G$-equivariant function on $P$, $C: P\to \mathfrak{g}$,
\begin{equation} \label{eq: function equiv}
\frac{\partial C^{\gamma}}{\partial y^{\alpha}}=\frac{1}{2}c^{\gamma}_{\beta\alpha}C^{\beta}.
\end{equation}

An Ehresmann connection $\Lambda$ on $P\to M$ defines a natural section $\delta_{\Lambda}$ of $J^1P^*\to\Pi_P$
\begin{equation*}
\delta_{\Lambda}(v_x\otimes\omega_y \otimes\vol)=(\omega_y\circ \Lambda )\wedge i_{v_x}\vol\in Z_y\cong (J^1P^*)_y, 
\end{equation*}
where $v_x\otimes\omega_y \otimes\vol\in T_xM\otimes V_y^*P\otimes\bigwedge^nT_x^*M$. From \eqref{pullback canform}, $\mathcal{H}=\delta-\delta_{\Lambda}$ is a map $\mathcal{H}:\Pi_P\to\bigwedge^nT^*M$. Hence, any Hamiltonian system can equivalently be described by a triple $(\Pi_P,\Lambda,\mathcal{H})$ and $\mathcal{H}$ is known as a \em Hamiltonian density. \em The local expression of \eqref{Hameq} are the  Hamilton--Cartan equations
\begin{equation} \label{Hamilton-Cartan}
\frac{\partial H}{\partial p^i_a}=\frac{\partial y^a}{\partial x^i}-\Lambda^a_i; \hspace{5mm}
-\frac{\partial H }{\partial y^a}=\frac{\partial p^i_a}{\partial x^i}+\frac{\partial \Lambda^b_i}{\partial y^a}p^i_b,
\end{equation}
where $\Lambda^a_i$ are the coefficients of the horizontal lift defined by $\Lambda$ and $\mathcal{H}=Hd^nx$. These equations can also be obtained in terms of a covariant bracket.

\begin{definition}
A differential form $F$ on $J^1P^*$ is  \textbf{horizontal} if $i_uF=0$ for any vertical tangent vector with respect to $J^1P^*\to M,$ that is, locally,
$$F=F_{i_1,\dots, i_r}dx^{i_1}\wedge\dots\wedge dx^{i_r}.$$
\end{definition}

\begin{definition} 
An horizontal $r$-form on $J^1P^*$ is \textbf{Poisson} if there is a vertical $(n-r)$-multivector field $\chi_F$ on $J^1P^*$ such that
\begin{equation} 
i_{\chi_F}\Omega =dF.
\end{equation}
\end{definition}
Given a Poisson $r$-form $F$ and a Poisson $s$-form $E$, their bracket is the $(r+s+1-n)$-form on $J^1P^*$,
\begin{equation}\label{Poisson bracket}
\{F,E\}=(-1)^{r(s-1)}i_{\chi_E}i_{\chi_F}\Omega,
\end{equation}
which, in turn, is a Poisson form. As seen in \cite{ForgerPoisson, PoissonForms}, this operator is a generalized graded Poisson bracket with a modified Leibniz rule. Poisson $(n-1)$-forms project from $J^1P^*$ to $\Pi_P$ and play a fundamental role in the formulation of Hamilton-Cartan equations \eqref{Hamilton-Cartan} in terms of a bracket, analogous to the role of affine functions in a cotangent bundle in Hamiltonian mechanics. 
\begin{theorem}\emph{\cite[Proposition 5.2]{someremarks}} \label{th: bracket formulation}
A section $p$ of $\Pi_P\to M$ is a solution of a given Hamiltonian system $(\Pi_P,\Lambda,\mathcal{H})$, $\mathcal{H}=H\vol$ if and only if for any horizontal Poisson $(n-1)$-form $F$ the following equation holds true:
\begin{equation} \label{Hameqs with bracket}
\{F,H\}\vol\circ p=d(p^*F)-(d^hF)\circ p,
\end{equation}
where $d^hF$ is the horizontal differential of $F$ with respect to the connection on $\Pi_P$.
\end{theorem}

The aforementioned connection on $\Pi_P\to M$ is defined from the Ehresmann connection $\Lambda$ on $P\to M$ and any linear connection on $M$. In local coordinates, the corresponding horizontal lift is
\begin{align} \label{eq: conn.poly}
\frac{\partial}{\partial x^i}\mapsto\frac{\partial}{\partial x^i}+\Lambda^a_i\frac{\partial}{\partial y^a}+\left(-\frac{\partial \Lambda^b_i}{\partial y^a}p^j_b+\Gamma^j_{ik}p^k_a-\Gamma^k_{ik}p^j_a\right)\frac{\partial}{\partial p^j_a},
\end{align}
where $\Gamma^k_{ij}$ are the Christoffel symbols of the linear connection.

Furthermore, the Poisson $(n-1)$-forms have a very particular structure. Let $W$ be any vector bundle over $P$, we introduce a linear operator between vector bundles
\begin{align*}
\llangle\cdot,\cdot\rrangle: \left(TM\otimes W^*\otimes \bigwedge^{n}T^*M\right)\otimes W \to& \bigwedge^{n-1}T^*M \\
\left(\frac{\partial}{\partial x^i}\otimes v\otimes \mathrm{v},w\right)\mapsto& \langle v, w\rangle i_{\partial /\partial x^i}\mathrm{v},
\end{align*}
and for any $\gamma\in\Gamma(W)$, the $(n-1)$-form $\theta_{\gamma}$ is defined on $q\in TM\otimes W^*\otimes \bigwedge^{n}T^*M$ by $(\theta_{\gamma})_q=\llangle q, \gamma \rrangle$.
\begin{proposition}\label{prop: n-1}\emph{\cite[Proposition 4.3]{someremarks}}
Any Poisson $(n-1)$-form on $\Pi_{P}$ can be written as
\begin{equation*}
F=\theta_X+\pi^*_{P,\Pi_P}\omega+\Upsilon,
\end{equation*}
where $X\in\mathfrak{X}^V(P)=\Gamma (VP)$ is a vertical vector field on $P$, $\omega$ is an horizontal $(n-1)$-form on $P$ and $\Upsilon$ is a closed horizontal $(n-1)$-form on $\Pi_P$.
\end{proposition}

\begin{remark}

The $(n-1)$-Poisson forms, as well as their structure given in the previous Proposition, match with the notion of current in \cite{AffineBracket}. The bracket and the Poisson equation in that work provide an alternative way of defining the dynamical equations that does not require to define a Hamiltonian density and horizontal derivatives with respect to a connection. The reduction analysis with this approach is an interesting matter to be carried out in the future.
\end{remark}

\section{Covariant Bracket on Affine Principal Bundles} \label{sec: covariant bracket}

Let $P=\left(Q \times_{M} E\right)$ be an affine principal bundle and consider the fiber coordinates $(x^i,y^\alpha)$ on $Q\to M$ and $(x^i,z^A)$ on $E\to M$. As stated in Section \ref{sec: multi}, this induces local coordinates $(x^i,y^\alpha,z^A,p^i_{\alpha},p^i_A, \tilde{p})$ on $J^1P^*$ and the local expression of the canonical $n+1$-form is:

\begin{equation}
    \Omega=dy^{\alpha} \wedge dp^{i}_\alpha \wedge d^{n-1} x_{i}+ dz^{A} \wedge dp^{i}_A \wedge d^{n-1} x_{i}-d \tilde{p} \wedge d^{n}x
\end{equation}

Let $\Lambda$ be a $G$-principal bundle on $P$ and $\mathcal{H}=Hd^nx$ be a Hamiltonian density on $P$, the Hamilton--Cartan equations are:

\begin{align}
\frac{\partial H}{\partial p^i_{\alpha}} &=\frac{\partial y^{\alpha}}{\partial x^{i}}-\Lambda_{i}^{\alpha}, \label{eq: HC12}  \qquad
\frac{\partial H}{\partial p^i_A} = \frac{\partial z^{A}}{\partial x^{i}}-\Lambda_{i}^{A},\\
\frac{\partial H}{\partial y^{\alpha}} &= -\left(\frac{\partial p^i_{\alpha}}{\partial x^{i}}+p^i_B\frac{\partial \Lambda_i^B}{\partial y^{\alpha}}+p^i_{\beta}\frac{\partial \Lambda_i^{\beta}}{\partial y^{\alpha}}\right), \label{eq: HCthird} \\
\frac{\partial H}{\partial z^A} &= -\left(\frac{\partial p^i_A}{\partial x^{i}}+p^i_B\frac{\partial \Lambda_i^B}{\partial z^A}+p^i_{\beta}\frac{\partial \Lambda_i^{\beta}}{\partial z^A}\right). \label{eq: HCFourth}
\end{align}

As stated in Theorem \ref{th: bracket formulation} and Proposition \ref{prop: n-1}, these equations can be equivalently obtained using Poisson $(n-1)$-forms, which are of the kind
\begin{equation}
F=\theta_X+\theta_Y+\pi^*_{P,\Pi_P}\omega+\Upsilon,
\end{equation}
where $X\in\mathfrak{X}^V(Q)=\Gamma (VQ)$ is a vertical vector field on $Q$, $Y\in\mathfrak{X}^V(E)=\Gamma (E)$, $\omega$ is an horizontal $(n-1)$-form on $P$ and $\Upsilon$ is a closed horizontal $(n-1)$-form on $\Pi_P$. The corresponding bracket is given by
\begin{equation} \label{eq: Bracket}
    \{F, H\}=\frac{\partial F^{i}}{\partial y^{\alpha}} \frac{\partial H}{\partial p_{\alpha}^i}+\frac{\partial F^{i}}{\partial z^{A}} \frac{\partial H}{\partial p_{A}^{i}}-\frac{\partial F^{i}}{\partial p_{\alpha}^{i}} \frac{\partial H}{\partial y^{\alpha}}-\frac{\partial F^{i}}{\partial p_{A}^{i}} \frac{\partial H}{\partial z^{A}}.
\end{equation}

\section{Reduction} \label{sec: reduction}
The main goal of this Section is to present the reduction of the Poisson equations \eqref{th: bracket formulation} when the Hamiltonian density on an affine principal bundle is invariant under the group $K$. The results represents the Poisson picture of the reduction developed in \cite{affinered} and constitute the core results of this paper.

\begin{proposition}\label{prop: red space} Let $P=Q\times_M E$ be an affine principal bundle, the reduced multisymplectic space $\Pi_{P}/ K$ is canonically isomorphic to
    \begin{equation} \label{eq: Hidentification}
        \Pi_{P}/ K\cong\left(T M \otimes \kk^{*} \otimes \bigwedge^{n} T^*M\right) \oplus \Pi_{E}
    \end{equation}
\end{proposition}
\begin{proof}
The $K$-action on $Q$, induces a natural action on $VQ\subset TQ$. Furthermore, the quotient is isomorphic to the adjoint bundle, $VQ/K\cong\kk$. Since the action on $E$ is trivial,
 \[
 V^{*}(Q \times E) / K=\left(V^{*} Q / K\right) \oplus\left(V^{*} E / K\right) =\kk^{*} \oplus V^{*} E.
 \]   
 The result follows as the action of $K$ on $\Pi_{P}$ is trivial on the factors $TM$ and $\bigwedge^{n} T^*M$.
\end{proof}

We denote by $\kappa:\Pi_{P}\rightarrow\Pi_{P}/K\cong(T M \otimes \kk^{*} \otimes \bigwedge^{n} T^*M)\oplus\Pi_E$ the natural projection introduced in Proposition \ref{prop: red space}. In addition, $\mu^i_\alpha$ and $\pi^i_A$ respectively denote multimomenta in $T M \otimes \kk^{*} \otimes \bigwedge^{n} T^*M)$ and $\Pi_E$.

\begin{definition} \label{def: affine poisson}
   An $(n-1)$-form $f$ on $\Pi_P/K$ is called \textbf{affine Poisson} if
   \begin{equation*}
        f=\theta_{\bar{\xi}}+\theta_Y+\pi^*_{M,\Pi_P/K}\omega+\Upsilon,
    \end{equation*}
    where $\xi\in\Gamma (\kk)$, $Y\in\Gamma (E)$, $\omega$ is an horizontal $(n-1)$-form on $M$, and $\Upsilon$ is a closed horizontal $(n-1)$-form on $\Pi_P/K$.
\end{definition}

\begin{lemma} \label{red. forms}
    The projection of an $K$-invariant Poisson $(n-1)$-form $F$ on $\Pi_P$ is an affine Poisson $(n-1)$-form $f$ on $\Pi_P/K$.
\end{lemma}
\begin{proof}
    From Proposition \ref{prop: n-1}, any $K$-invariant Poisson $(n-1)$-form on $\Pi_P$ is of the kind
    \begin{equation*}
        F=\theta_X+\theta_Y+\pi^*_{P,\Pi_P}\omega+\Upsilon,
    \end{equation*}
    where $X\in\Gamma (VQ/K)\cong\Gamma(\kk^*)$, $Y\in\Gamma (E)$, $\omega\in\Omega^{n-1}(M)$, and $Z$ is a closed $K$-invariant $(n-1)$-form on $\Pi_P$.
\end{proof}

%No he definido Hamiltonian density on the reduced space.

\begin{definition}
Let $f$ be an affine Poisson $(n-1)$-form on $\Pi_P/K$ and $h$ a Hamiltonian density on $\Pi_{P}/ K\cong(T M \otimes \kk^{*} \otimes \bigwedge^{n} T^*M) \oplus \Pi_{E}$, their bracket is
\begin{equation} \label{eq: Reduced Bracket}
\{f,h\}=\{\bar{\xi},h\}_{\mathrm{LP}}+ \{f,h\}_{E},
\end{equation}
where $\{f,h\}_{E}$ is the Poisson bracket on $\Pi_E$, and 
\begin{equation} \label{eq: Lie-Poisson bracket}
\{\bar{\xi},h\}_{\mathrm{LP}}=-\left\langle\mu,\left[\bar{\xi},\frac{\delta h}{\delta\mu}\right]\right\rangle
\end{equation}
is the Lie-Poisson bracket on $T M \otimes \kk^{*} \otimes \bigwedge^{n} T^*M.$
\end{definition}
 
\begin{proposition}\label{kappa.is.poisson}
Let $\kappa:\Pi_{P}\rightarrow\Pi_{P}/K\cong(T M \otimes \kk^{*} \otimes \bigwedge^{n} T^*M)\oplus\Pi_E$ be the natural projection in Proposition \ref{prop: red space}, for any $K$-invariant Poisson $(n-1)$-form $F$ and any $K$-invariant Poisson function $H$ on $\Pi_P/K$
\begin{equation}
\{F,H\}=\kappa^*\{f,h\},
\end{equation}
where the bracket on the left hand side is defined by \eqref{Poisson bracket}, $f$ is an affine Poisson $(n-1)$-form on $\Pi_{P}/K$ such that $\kappa^*f=F$, $h$ is a function such that $\kappa^*h=H$, and the bracket on the right hand side is defined by \eqref{eq: Reduced Bracket}.
\end{proposition}

\begin{proof}
As $F$ and $H$ are $K$-invariant, applying Equation \eqref{eq: function inv} on the expression \eqref{eq: Bracket} provides a function $r$ on $\Pi_{P}/K$,
$$
r= -\frac{1}{2} \mu^j_{\gamma} c^{\gamma}_{\beta \alpha} \frac{\partial f^i}{\partial \mu^j_\beta} \cdot \frac{\partial h}{\partial \mu_\alpha^i}+\frac{\partial f^i}{\partial \mu^i_\alpha}\cdot\frac{1}{2} \mu^j
_{\gamma} c^{\gamma}_{\beta \alpha} \frac{\partial h}{\partial \mu_\beta^j}+\frac{\partial f^i}{\partial z^A} \cdot \frac{\partial h}{\partial \pi^{i}_A}-\frac{\partial f^i}{\partial \pi^i_A} \cdot \frac{\partial h}{\partial z^A}
$$
such that $\{F,H\}=\kappa^*r$. We shall see that $r=\{f,h\}$ as desired. Since $f$ is the projection of $F$, from Lemma \ref{red. forms}, $f$ is an affine Poisson form and $\partial f^{i}/\partial \mu^j_{\alpha}=\xi^{\alpha} \delta_{j}^{i}$. Therefore,
\begin{align*}
    -\frac{1}{2} \mu^j_{\gamma} c^{\gamma}_{\beta \alpha} \frac{\partial f^i}{\partial \mu^j_\beta} \frac{\partial h}{\partial \mu_\alpha^i} 
    &+ \frac{\partial f^i}{\partial \mu^i_\alpha} \frac{1}{2} \mu^j_{\gamma} c^{\gamma}_{\beta \alpha} \frac{\partial h}{\partial \mu_\beta^j} \\
    &= \frac{1}{2} \mu_{\gamma}^{j} c_{\beta \alpha}^{\gamma} \left(\xi^{\alpha} \frac{\partial h}{\partial \mu_{\beta}^{j}} - \xi^{\beta} \frac{\partial h}{\partial \mu^{j}_\alpha} \right) 
    = \mu_{\gamma}^{j} c_{\beta \alpha}^{\gamma} \xi^{\alpha} \frac{\partial h}{\partial \mu^{j}_\beta},
\end{align*}

%\begin{multline*}
%    -\frac{1}{2} \mu^j_{\gamma} c^{\gamma}_{\beta \alpha} \frac{\partial f^i}{\partial \mu^j_\beta} \frac{\partial h}{\partial \mu_\alpha^i} 
%    + \frac{\partial f^i}{\partial \mu^i_\alpha} \frac{1}{2} \mu^j_{\gamma} c^{\gamma}_{\beta \alpha} \frac{\partial h}{\partial \mu_\beta^j} \\
%    = \frac{1}{2} \mu_{\gamma}^{j} c_{\beta \alpha}^{\gamma} \left(\xi^{\alpha} \frac{\partial h}{\partial \mu_{\beta}^{j}} - \xi^{\beta} \frac{\partial h}{\partial \mu^{j}_\alpha} \right) 
%    = \mu_{\gamma}^{j} c_{\beta \alpha}^{\gamma} \xi^{\alpha} \frac{\partial h}{\partial \mu^{j}_\beta},
%\end{multline*}
and 
\begin{equation} \label{eq: local bracket}
    r=-\mu_{\gamma}^{j} c_{\alpha\beta }^{\gamma} \xi^{\alpha} \frac{\partial h}{\partial \mu^{j}_\beta} +\frac{\partial f^i}{\partial z^A} \cdot \frac{\partial h}{\partial \pi^{i}_A}-\frac{\partial f^i}{\partial \pi^i_A} \cdot \frac{\partial h}{\partial z^A},
\end{equation}
which is the local expression of $\{f,h\}$.
\end{proof}

\begin{proposition} \label{dh reduction}
Let $F$ be a $K$-invariant Poisson $(n-1)$-form on $\Pi_P$ and $f$ the affine Poisson $(n-1)$-form on $\Pi_P/K$ such that $\kappa^*f=F$, then
\begin{equation}
    d^hF= \kappa^*\left(d^hf\right),
\end{equation}
where $d^hF$ and $d^hf$ are the horizontal differentials with respect to the connection on $\Pi
_P$ and $\Pi_E$ introduced in \eqref{eq: conn.poly}, as well as the induced connection on $T M \otimes \kk^{*} \otimes \bigwedge^{n} T^*M$.
\end{proposition}
\begin{proof}
The local expression of $d^hF$ is
$$
\begin{aligned}
   \frac{\partial F^i}{\partial x^i}+\frac{\partial F^i}{\partial y^{\alpha}} \Lambda_i^{\alpha}+\frac{\partial F^i}{\partial z^A} \Lambda_i^A
+\frac{\partial F^i}{\partial p_\alpha^j}\left(-\frac{\partial \Lambda^\beta_i}{\partial y^\alpha} p_\beta^j-\frac{\partial \Lambda_i^B}{\partial y^\alpha}p^j_B+\Gamma_{i k}^j p_\alpha^k-\Gamma_{i k}^k p_\alpha^j\right)\phantom{.}&\\
+\frac{\partial F^i}{\partial p_A^j}\left(-\frac{\partial \Lambda_i^B}{\partial z^A}p^j_B-\frac{\partial \Lambda_i^\beta}{\partial z^A} p_\beta^j+\Gamma_{i k}^j p_A^k-\Gamma_{i k}^k p_A^j\right).& 
\end{aligned}
$$
Since $\Lambda$ is $K$-invariant, a direct application of Equations \eqref{eq: function inv} and \eqref{eq: function equiv} implies that,
\begin{equation*}
\frac{\partial \Lambda^B_i}{\partial y^{\alpha}}=0, \qquad
\frac{\partial \Lambda^{\gamma}_i}{\partial y^{\alpha}}=\frac{1}{2}c^{\gamma}_{\beta\alpha}\Lambda^{\beta}_i.
\end{equation*}
Then, $d^hF=\kappa^*q$, where 

\begin{align}
   q &= \frac{\partial f^i}{\partial x^i} + \frac{\partial f^i}{\partial z^A} \Lambda_i^A 
   + \frac{1}{2} \mu^j_{\gamma} c^{\gamma}_{\beta\alpha} \frac{\partial f^i}{\partial \mu^j_{\beta}} \Lambda_i^{\alpha} \notag \\
   &\quad + \frac{\partial f^i}{\partial \mu_\alpha^j} 
   \left(-\frac{1}{2} c^{\beta}_{\gamma\alpha} \Lambda_i^\gamma \mu_\beta^j 
   + \Gamma_{i k}^j \pi_\alpha^k - \Gamma_{i k}^k \pi_\alpha^j \right) \notag \\
   &\quad + \frac{\partial f^i}{\partial \pi_A^j} 
   \left(-\frac{\partial \Lambda_i^B}{\partial z^A} \pi^j_B 
   - \frac{\partial \Lambda_i^\beta}{\partial z^A} \mu_\beta^j 
   + \Gamma_{i k}^j \pi_A^k - \Gamma_{i k}^k \pi_A^j \right) \notag \\
   &= \frac{\partial f^i}{\partial x^i} + \frac{\partial f^i}{\partial z^A} \Lambda_i^A \notag \\
   &\quad + \frac{\partial f^i}{\partial \mu_\alpha^j} 
   \left(-\mu_\gamma^j c^{\gamma}_{\beta\alpha} \Lambda_i^\beta 
   + \Gamma_{i k}^j \pi_\alpha^k - \Gamma_{i k}^k \pi_\alpha^j \right) \notag \\
   &\quad + \frac{\partial f^i}{\partial \pi_A^j} 
   \left(-\frac{\partial \Lambda_i^B}{\partial z^A} \pi^j_B 
   - \frac{\partial \Lambda_i^\beta}{\partial z^A} \mu_\beta^j 
   + \Gamma_{i k}^j \pi_A^k - \Gamma_{i k}^k \pi_A^j \right),\notag
\end{align}
and the local expression of $d^hf$ is identified.
\end{proof}

\begin{theorem} \label{reduction}
Let $P = Q \times_M E \to M$ be the affine principal bundle defined by $Q\to M$ and a representation of $K$ on $V$. Let $\vol$ be a volume form on $M$, $\Lambda$ a $G=K \ltimes V$ principal connection on $P\to M$ and $\mathcal{H}=H\vol$ a $K$-invariant Hamiltonian density on $\Pi_P$. Denote by $\mathcal{h}=h\vol$ the reduced Hamiltonian density and for any section $p$ of $\Pi_P\to M$, let $\mu\oplus\pi=\kappa \circ p$ be the reduced section of
$$\Pi_P/K\cong\left(TM\otimes\g^*\otimes\bigwedge^{n-1}T^*M\right)\oplus\Pi_E\to M.$$
Then, the following are equivalent:
\begin{enumerate}[\em (i)\em]
\item for every Poisson $(n-1)$-form $F$ on $\Pi_P$, the following identity holds true:
\begin{equation*}
\{F,H\}\vol=d(F\circ p)-d^{h}F\circ p.
\end{equation*}
\item the section $p$ satisfies the Hamilton--de Donder equations,
\item for every affine Poisson $(n-1)$-form $f$ on $\Pi_P/K$,
\begin{equation}\label{intrinsic.red.eq}
\{f,h\}\vol=d(f\circ(\mu\oplus\pi))-d^{h}f\circ(\mu\oplus\pi
),
\end{equation}
where the bracket is defined by Equation \eqref{eq: Reduced Bracket}.
\item the section $\mu$ satisfies the Lie--Poisson equations
\begin{equation}\label{eq: Lie-Poisson}
  \mathrm{div}^{\Lambda}\mu =\mathrm{ad}^*_{\partial h /\partial\mu}\mu,
\end{equation}
and $\pi$ satisfies the Hamilton--de Donder equations.
\end{enumerate}
\end{theorem}
\begin{proof}
The equivalence (i)$\Leftrightarrow$(ii) is stated in Theorem \ref{th: bracket formulation}. We shall now prove the equivalence (i)$\Leftrightarrow$(iii). From Propositions \ref{kappa.is.poisson} and \ref{dh reduction}, as every affine Poisson $(n-1)$-form $f$ can be obtained by reduction of a Poisson $(n-1)$-form $F$ on $\Pi_P$, we only need to prove that  $d(F\circ p)=\kappa^*(d(f\circ(\mu\oplus\pi)))$. Indeed, since $F$ is horizontal, for any $v_1,\dots v_{n-1}$ vectors in $T_xM$
\begin{align*}
(p^*F)_x(v_1,\dots,v_{n-1})&=F_{p(x)}(T_xp(v_1),\dots,T_xp(v_{n-1}))=F_{p(x)}(v_1,\dots,v_{n-1})\\&=(F\circ p)_x(v_1,\dots,v_{n-1}).
\end{align*}
Similarly, $(\mu\oplus\pi)^*f=f\circ(\mu\oplus\pi)$. Thus,
$$F\circ p=p^*F=p^*\kappa^*f=(\mu\oplus\pi)^*f=f\circ(\mu\oplus\pi),$$
and $d(F\circ p)=d(f\circ(\mu\oplus\pi))$.

Finally, equivalence (iii)$\Leftrightarrow$(iv) is obtained as follows:

\begin{align*}
d(f\circ(\mu\oplus&\pi))-d^{h}f\circ(\mu\oplus\pi
)\\
    =&\left(\frac{\partial f^{i}}{\partial x^i}
    +\frac{\partial f^{i}}{\partial z^{A}}\frac{\partial z^{A}}{\partial x^{i}}
    +\frac{\partial f^{i}}{\partial \mu^j_{\alpha}}\frac{\partial \mu^j_{\alpha}}{\partial x^{i}}
    +\frac{\partial f^{i}}{\partial \pi^j_{A}}\frac{\partial \pi^j_{A}}{\partial x^{i}} \right.\\
    &-\frac{\partial f^{i}}{\partial x^i}
    -\frac{\partial f^i}{\partial z^A} \Lambda_i^A
    -\frac{\partial f^i}{\partial \mu_\alpha^j} 
    \left(-\mu_\gamma^j c^{\gamma}_{\beta\alpha} \Lambda_i^\beta 
    + \Gamma_{i k}^j \pi_\alpha^k - \Gamma_{i k}^k \pi_\alpha^j \right)\\
    &\left. - \frac{\partial f^i}{\partial \pi_A^j} 
   \left(-\frac{\partial \Lambda_i^B}{\partial z^A} \pi^j_B 
   - \frac{\partial \Lambda_i^\beta}{\partial z^A} \mu_\beta^j 
   + \Gamma_{i k}^j \pi_A^k - \Gamma_{i k}^k \pi_A^j \right)\right)\vol\\
   =&\left(\frac{\partial f^{i}}{\partial z^{A}} 
    \left(\frac{\partial z^{A}}{\partial x^{i}} - \Lambda_i^A \right)
    +\frac{\partial f^{i}}{\partial \mu^j_{\alpha}} 
    \left(\frac{\partial \mu^j_{\alpha}}{\partial x^{i}} 
    +\mu_\gamma^j c^{\gamma}_{\beta\alpha} \Lambda_i^\beta 
    - \Gamma_{i k}^j \pi_\alpha^k + \Gamma_{i k}^k \pi_\alpha^j \right)\right.\\
    &\left.+\frac{\partial f^{i}}{\partial \pi^j_{A}} 
    \left(\frac{\partial \pi^j_{A}}{\partial x^{i}} 
    +\frac{\partial \Lambda_i^B}{\partial z^A} \pi^j_B 
    +\frac{\partial \Lambda_i^\beta}{\partial z^A} \mu_\beta^j 
    - \Gamma_{i k}^j \pi_A^k + \Gamma_{i k}^k \pi_A^j \right)\right)\vol.
\end{align*}
As $f$ is an affine Poisson $(n-1)$-form, $ \partial f^{i}/\partial \mu_{\alpha}=\xi^{\alpha} \delta_{j}^{i}$; 
$\partial f^{i}/\partial \pi_{A}^{j}=Y^{A} \delta_{j}^{i}$,

\begin{align} \label{eq: right side}
d(f\circ(\mu\oplus\pi))-d^{h}f\circ(\mu\oplus\pi) 
&= \frac{\partial f^{i}}{\partial z^{A}}\left(\frac{\partial z^{A}}{\partial x^{i}} -\Lambda_{i}^{A}\right)  
+ \xi^{\alpha} \left(\frac{\partial \mu_{\alpha}^{i}}{\partial x^{i}}-\mu_{\gamma}^{i} c_{\beta \alpha}^{\gamma} \Lambda_{i}^{\beta}\right) \notag \\
&\quad + Y^{A} \left(\frac{\partial \pi^{i}_A}{\partial x^{i}}-\frac{\partial \Lambda_{i}^{B}}{\partial z^{A}} \pi_{B}^{i}-\frac{\partial \Lambda_{i}^{\beta}}{\partial z^{A}} \mu_{\beta}^{i}\right).
\end{align}
Comparison of Equation \eqref{eq: right side} with the local expression of $\{f,h\}$ in \eqref{eq: local bracket} provides the reduced equations:
\begin{align}
    \frac{\partial h}{\partial \pi^{i}_A} 
    &= \frac{\partial z^{A}}{\partial x^{i}} - \Lambda_{i}^{A}, \label{eq: RHC PI}
    \\
    \frac{\partial h}{\partial z^{A}} 
    &= -\left( 
        \frac{\partial \pi_{A}^{i}}{\partial x^{i}} 
        - \frac{\partial \Lambda_{i}^{B}}{\partial z^{A}} \pi_{B}^{i} 
        - \frac{\partial \Lambda_{i}^{\beta}}{\partial z^{A}} \mu_{\beta}^{i}
    \right),
    \label{eq: RHC Z}
    \\
    \mu_{\gamma}^{j} c_{\beta \alpha}^{\gamma} 
    \frac{\partial h}{\partial \mu^{i} _{\beta}} 
    &= \frac{\partial \mu_{\alpha}^{i}}{\partial x^{i}} 
    + \mu_{\gamma}^{i} c_{\beta \alpha}^{\gamma} \Lambda_{i}^{\beta}.
    \label{eq: RHC MU}
\end{align}
\end{proof}

\section{Application to molecular strands} \label{sec: example}
The reduction procedure developed in the previous Sections provides a fundamental framework for analyzing $G$-strands, introduced in \cite{matrixgstrands,gstrands}, when $G$ is a semi-direct product. In this Section, we illustrate this methodology to study Molecular Strands, a physically relevant example comprehensively discussed in \cite{MolStrand}.

To describe this system, we consider a trivial principal bundle \( Q = \mathbb{R}^2 \times SO(3) \to \mathbb{R}^2 \)  with structure group \( K = SO(3) \). The vector bundle \( V = \mathbb{R}^3 \) naturally acts on \( V \), leading to the associated bundle \( E = \mathbb{R}^2 \times \mathbb{R}^3 \to \mathbb{R}^2 \). The resulting affine principal bundle is \( P = \mathbb{R}^2 \times SE(3) \to \mathbb{R}^2 \) with structure group \( SE(3) = SO(3) \ltimes \mathbb{R}^3 \), the group of Euclidean affine isometries of \( \mathbb{R}^3 \). Within this setup, we can take local coordinates \( (s,t) \) in \( \mathbb{R}^2 \) (\( s \) for space and \( t \) for time), identify sections of $Q$ with maps \( R: \mathbb{R}^2 \to SO(3) \) and sections of $E$ with mappings \( \rho: \mathbb{R}^2 \to \mathbb{R}^3 \). We can also identify \( \kk=\mathfrak{so}(3) \) with \( \mathbb{R}^3 \) in the standard way:

\[
\begin{pmatrix}
0 & c & b \\
-c & 0 & a \\
-b & -a & 0
\end{pmatrix}
\mapsto (a, b, c),
\]
and the Lie bracket with the cross product, \( \times \), in \( \mathbb{R}^3 \). Then, from Proposition \ref{prop: red space} we know that,
\[
\Pi_P/K\cong\left(T M \otimes \R^3 \otimes \bigwedge^{n} T^*M\right) \oplus \Pi_{E}
\]
and multimomenta can be described by $\R^3$-valued $1$-forms on $\R^2$:
\[
\mu = \frac{\partial}{\partial s} \otimes \mu^s \otimes d s \wedge d t + \frac{\partial}{\partial t} \otimes \mu^t \otimes d s \wedge d t = \mu^s \otimes d t -\mu^t \otimes ds 
\]
\[
\pi = \frac{\partial}{\partial s} \otimes \pi^s \otimes d s \wedge d t + \frac{\partial}{\partial t} \otimes \pi^t \otimes d s \wedge d t = \pi^s \otimes d t -\pi^t \otimes ds.
\]

Consider that the plane $\R^2$ is equipped with the Minkowski metric $g=ds\otimes ds-v^2 dt\times dt$. We introduce the Hamiltonian $h:\left(T M \otimes \R^3 \otimes \bigwedge^{n} T^*M\right) \oplus \Pi_{E}\to \R$ defined by:
\begin{align}
h\left(\mu^s, \mu^t, \rho, \pi^s, \pi^t\right) &= \frac{1}{2} \left\langle \pi^s, \pi^s \right\rangle 
- \frac{v^2}{2} \left\langle \pi^t, \pi^t \right\rangle 
+ U(\|\rho\|) \notag \\
&\quad + \frac{1}{2} \left\langle \mu^t - \rho \times \pi^t, I^{-1} \left(\mu^t - \rho \times \pi^t \right) \right\rangle \notag \\
&\quad - \frac{1}{2} \left\langle \mu^s - \rho \times \pi^s, J^{-1} \left(\mu^s - \rho \times \pi^s \right) \right\rangle.
\end{align}
This represents a minimal coupling of the Hamiltonian of a string in a radial potential $U(\|\rho\|)$ and the $SO(3)$ chiral model studied in \cite{gstrands}. This system was discussed in \cite{affinered} from a Lagrangian point of view. A direct calculation shows that:

\[
\frac{\delta h}{\delta \mu^s}=-J^{-1}\left(\mu^s-\rho \times \pi^s\right),
\quad
\frac{\delta h}{\delta \mu^t} = I^{-1} \left( \mu^t - \rho \times \pi^t \right),
\]

\[
\frac{\delta h}{\delta \pi^s} = \pi^s - \rho \times \left( J^{-1} \left( \mu^s - \rho \times \pi^s \right) \right),
\quad 
\frac{\delta h}{\delta \pi^t} = -v^2 \pi^t + \rho \times \left( I^{-1} \left( \mu^t - \rho \times \pi^t \right) \right)
\]

\[
\frac{\delta h}{\delta \rho} = \pi^s \times \left( J^{-1} \left( \mu^s - \rho \times \pi^s \right) \right) - \pi^t \times \left( I^{-1} \left( \mu^t - \rho \times \pi^t \right) \right) + \frac{\partial U}{\partial \|\rho\| }\frac{\rho}{\|\rho\|}
\]
%\[
%\frac{\delta h}{\delta \mu} = \left( \frac{\partial}{\partial t} \right) \otimes \left( -J^{-1} \left( \mu^s - \rho \times \pi^s \right) \right) + \left( \frac{\partial}{\partial s} \right) \otimes \left( I^{-1} \mu^t - \rho \times \pi^t \right)
%\]
From the above expressions we obtain the reduced Hamilton-Cartan equations of the Molecular Strand. Equations \eqref{eq: RHC PI} take the form, 
\begin{align}
&\frac{\partial \rho}{\partial s}=\pi^{s}-\rho \times\left(J \left(\mu^{s}-\rho \times \pi^{s}\right)\right), \label{eq: Mol.Strand.ros} \\
& \frac{\partial \rho}{\partial t}=-v^{2} \pi^{t}+\rho \times\left(I^{-1}\left(\mu^{t}-\rho \times \pi^{t}\right)\right), \label{eq: Mol.Strand.rot}
\end{align}
similarly, Equation \eqref{eq: RHC Z} transforms into,
\begin{multline} \label{eq: Mol.Strand.pi}
    \partial_{s} \pi^{s}+\partial_{t} \pi^{t} 
    - \pi^{s} \times \left(J^{-1} \left(\mu^{s} - \rho \times \pi^{s}\right)\right) \\
    + \pi^{t} \times \left(I^{-1}\left(\mu^{t}-\rho \times \pi^{t}\right)\right) 
    = -\frac{\partial U}{\partial \|\rho\| } \frac{\rho}{\|\rho\|},
\end{multline}
while Equation \eqref{eq: RHC MU} particularizes to:
\begin{equation} \label{eq: Mol.Strand.mu}
  \partial_s \mu^{s}+\partial_{t} \mu^{t}=-\mu^{s} \times\left(J^{-1}\left(\mu^{s}-\rho \times \pi^s\right)\right)+\mu^{t} \times\left(I^{-1}\left(\mu^{t}-\rho \times \pi^{t}\right)\right).  
\end{equation}
These equations of motion coincide with those obtained in \cite{affinered} from a Lagrangian perspective. Indeed, Equations \eqref{eq: Mol.Strand.ros} and \eqref{eq: Mol.Strand.rot} relate the multimomenta and the evolution of $\rho(s,t)$, while Equation \eqref{eq: Mol.Strand.pi} and \eqref{eq: Mol.Strand.mu} are respectively equivalent to $\quad \square_{\Omega, \omega} \rho =-2 U^{\prime}(\langle \rho, \rho \rangle)\rho$ and $\mathrm {div}^{\sigma}\delta l/\delta \sigma=0$ in \cite{affinered}.

\bibliography{references}{}
\bibliographystyle{abbrvbf}
\end{document}